\documentclass[12pt, reqno, twoside, letterpaper]{amsart}

\usepackage{paperstyle}

\newcommand{\be}{\begin{equation}}
\newcommand{\ee}{\end{equation}}
\newcommand{\dalign}[1]{\[\begin{aligned} #1 \end{aligned}\]}

\newcommand{\euB}{\EuScript{B}}

\newcommand{\er}{\mathrm{e}}

\newcommand{\one}{\mathbb{1}}
\newcommand{\ochi}{{\overline\chi}}
\newcommand{\opsi}{{\overline\psi}}
\newcommand{\otheta}{{\overline\vartheta}}

\title[GRH from zeros of a single $L$-function]
{The Generalized Riemann Hypothesis\\ from zeros of
a single $L$-function}
      
\author[W.~Banks]{William Banks}

\address{Department of Mathematics, 
         University of Missouri, 
         Columbia MO 65211, USA.}

\email{bankswd@missouri.edu}
        
\date{\today}

\begin{document}

\begin{abstract}
For each primitive Dirichlet character $\chi$,
a hypothesis ${\tt GRH}^\dagger[\chi]$ is formulated
in terms of zeros of the associated $L$-function $L(s,\chi)$.
It is shown that for any such character, 
${\tt GRH}^\dagger[\chi]$ is equivalent to the Generalized
Riemann Hypothesis.
\end{abstract}



\thanks{Data Availability Statement: Data sharing not applicable to this article as no datasets
were generated or analysed during the current study.}

\thanks{Potential Conflicts of Interest: NONE}

\thanks{Research Involving Human Participants and/or Animals: NONE}

\maketitle



\centerline{\it Dedicated to Hugh Montgomery and Bob Vaughan}

\vskip0.4in
\section{Introduction}
\label{sec:intro}

An old result of Sprind\v zuk~\cite{Sprind1,Sprind2}
(which he obtained by developing ideas of Linnik~\cite{Linnik})
states that, under the Riemann Hypothesis (RH), the Generalized Riemann
Hypothesis (GRH) holds for \emph{all} Dirichlet $L$-functions provided some
suitable conditions on the vertical distribution of the zeros
of $\zeta(s)$ are met. More precisely, the \emph{Linnik–Sprind\v zuk theorem}
asserts that GRH is equivalent to the validity of both RH and the hypothesis that,
for any rational number $\xi\defeq h/k$ with $0<|h|\le k/2$ and $(h,k)=1$,
and any real $\eps>0$, the bound
\[
\sum_{\rho=\frac12+i\gamma}|\gamma|^{i\gamma}
\er^{-i\gamma-\pi|\gamma|/2}(x+2\pi i\xi)^{-\rho}
+\frac{\mu(k)}{\phi(k)}\frac{1}{x\sqrt{2\pi}}
\,\mathop{\ll}\limits_{\xi,\eps}\,x^{-1/2-\eps}
\]
holds for $x\to 0^+$. Similar results 
have been attained by Fujii~\cite{Fujii1,Fujii2,Fujii3},
Suzuki~\cite{Suzuki}, Kaczorowski and Perelli~\cite{KacPer},
and the author~\cite{Banks1,Banks2}.

In the present paper, we establish an analogous result in which
$\zeta(s)$ is replaced by the Dirichlet $L$-function $L(s,\chi)$
attached to an arbitrary primitive character $\chi$.
To formulate the theorem, we introduce some notation.
In what follows, $C_c^\infty(\R^+)$ denotes the space of smooth
functions $\euB:\R^+\to\C$ with compact support in $\R^+$.
As usual, we write $\e(u)\defeq\er^{2\pi iu}$ for all $u\in\R$.
Let $\chi$ be a primitive character modulo $q$, and put
\be\label{eq:red}
\kappa_\chi\defeq\begin{cases}
0&\quad\hbox{if $\chi(-1)=+1$,}\\
1&\quad\hbox{if $\chi(-1)=-1$,}\\
\end{cases}
\quad\tau(\chi)\defeq\sum_{a\bmod q}\chi(a)\e(a/q),
\quad\epsilon_\chi\defeq\frac{\tau(\chi)}{i^{\kappa_\chi}\sqrt{q}}.
\ee
We recall the asymmetric form of the functional equation
\[
L(s,\chi)=\cX_\chi(s)L(1-s,\ochi),
\]
where
\be\label{eq:barchetta}
\cX_\chi(s)\defeq\epsilon_\chi2^s\pi^{s-1}q^{1/2-s}
\Gamma(1-s)\sin\tfrac{\pi}{2}(s+\kappa_\chi).
\ee
In particular, if $\one$ is
 the trivial character given by $\one(n)=1$
for all $n\in\Z$, then the functional equation
$\zeta(s)=\cX_{\one}(s)\zeta(1-s)$ holds with
\[
\cX_{\one}(s)\defeq 2^s\pi^{s-1}\Gamma(1-s)\sin\tfrac{\pi s}{2}.
\]

Finally, for a fixed primitive character $\chi$ modulo $q$, consider
the following two hypotheses concerning the zeros of $L(s,\chi)$.
The first hypothesis is

\bigskip\noindent{\sc Hypothesis ${\tt GRH}[\chi]$}:
{\it If $L(\beta+i\gamma,\chi)=0$ and $\beta>0$, then $\beta=\tfrac12$.}

\bigskip\noindent Note that GRH is equivalent to the assertion
that ${\tt GRH}[\chi]$ holds for all primitive
characters $\chi$.  The second (stronger) hypothesis is

\bigskip\noindent{\sc Hypothesis ${\tt GRH}^\dagger[\chi]$}:
{\it Hypothesis ${\tt GRH}[\chi]$ is true, and for any rational number
$\xi\defeq h/k$ with $h,k>0$
and $(h,k)=1$, any $\euB\in C_c^\infty(\R^+)$, and any $\eps>0$, 
the bound
\be\label{eq:superbound}
\ssum{\rho=\frac12+i\gamma}\xi^{-\rho}
\cX_\ochi(1-\rho)\euB\Big(\frac{\gamma}{2\pi\xi X}\Big)
+C_\euB\widetilde C_{\chi,\xi}X
\,\mathop{\ll}\limits_{\chi,\xi,\euB,\eps}\,X^{1/2+\eps}
\ee
holds for $X\to\infty$, where
\be\label{eq:Cconsts}
C_\euB\defeq\int_0^\infty\euB(u)\,du,\qquad
\widetilde C_{\chi,\xi}\defeq\begin{cases}
\displaystyle{\frac{\ochi(h)\chi(k)\mu(k)\,q}{\phi(qk)}}
&\quad\hbox{if $(h,q)=1$},\\
0&\quad\hbox{otherwise},\\
\end{cases} 
\ee
the sum in \eqref{eq:superbound} runs over the complex zeros
$\rho=\frac12+i\gamma$ of $L(s,\chi)$ $($counted with multiplicity$)$,
and the implied constant in \eqref{eq:superbound} depends only on
$\chi$, $\xi$, $\euB$, and $\eps$.
}

\bigskip Our aim in this note is to prove the following theorem.

\begin{theorem}\label{thm:RHvsGRH} For each primitive
character $\chi$, the hypotheses ${\tt GRH}^\dagger[\chi]$ 
and ${\rm GRH}$ are equivalent.\end{theorem}

Theorem~\ref{thm:RHvsGRH} generalizes the main result of \cite{Banks2},
in which the author showed that ${\tt GRH}^\dagger[\one]$
and ${\rm GRH}$ are equivalent.

\begin{corollary}\label{cor:RHvsGRH}
If ${\tt GRH}^\dagger[\chi]$ holds for one primitive
character $\chi$, then it holds for all primitive
characters.
\end{corollary}

We emphasize that ${\tt GRH}^\dagger[\chi]$ is formulated
entirely in terms of the zeros of a single $L$-function
$L(s,\chi)$.
If $\chi$ and $\psi$ are different primitive characters,
it is reasonable to predict that the zeros of $L(s,\chi)$ and $L(s,\psi)$
are unrelated, and that there is no reason
\emph{a priori} that the hypotheses ${\tt GRH}^\dagger[\chi]$ and
${\tt GRH}^\dagger[\psi]$ should be connected. Nevertheless,
${\tt GRH}^\dagger[\chi]$ and ${\tt GRH}^\dagger[\psi]$
are actually equivalent in view of Corollary~\ref{cor:RHvsGRH}.

\section{Preliminaries}

We continue to use the notation introduced in \S\ref{sec:intro}.
Below, $\chi$ always denotes 
an arbitrary (but fixed) primitive Dirichlet
character of modulus $q\ge 1$.

Throughout the paper,  implied constants in the
symbols $\ll$, $O$, etc., may depend on various parameters
as indicated by the notation (see, e.g., \eqref{eq:superbound}), but
such constants are independent of all other parameters.

\subsection{The function $\cX_\chi(s)$}

\begin{lemma}\label{lem:Xxexpansion}
Let $\cI$ be a bounded interval in $\R$. Uniformly for $c\in\cI$ and $t\ge 1$,
we have
\be\label{eq:spray}
\cX_\chi(1-c-it)=\tau(\chi) q^{c-1}\er^{-\pi i/4}
\exp\Big(it\log\Big(\frac{qt}{2\pi\er}\Big)\Big)
\Big(\frac{t}{2\pi}\Big)^{c-1/2}\big\{1+O_\cI(t^{-1})\big\}.
\ee
\end{lemma}

\begin{proof}
Replacing $s$ by $1-s$ in \eqref{eq:barchetta} gives
\[
\cX_\chi(1-s)=\epsilon_\chi 2^{1-s}\pi^{-s}q^{s-1/2}
\Gamma(s)\sin\tfrac{\pi}{2}(1-s+\kappa_\chi).
\]
Let $s=c+it$, $t\ge 1$. Using Stirling's formula for the gamma function
\[
\Gamma(s)=\sqrt{2\pi}\,s^{s-1/2}\er^{-s}\{1+O(t^{-1})\}
\]
(see, e.g., Montgomery and Vaughan~\cite[Theorem~C.1]{MontVau})
along with the estimates
\[
(s-\tfrac12)\log s=(c-\tfrac12)\log t+c
+(t\log t-\tfrac{\pi}{4})i+\tfrac{\pi is}{2}+O(t^{-1})
\]
and
\[
\sin\tfrac{\pi}{2}(1-s+\kappa_\chi)
=\tfrac12i^{\kappa_\chi}\er^{-\pi i s/2}\{1+O(\er^{-\pi t})\},
\]
and recalling \eqref{eq:red}, a straightforward computation
leads to \eqref{eq:spray}.
\end{proof}

The following lemma is due to Gonek~\cite[Lemma~2]{Gonek};
the proof is based on the stationary phase method.

\begin{lemma}\label{lem:gonekLem8}
Uniformly for $c\in[\frac{1}{10},2]$
and $a<b\le 2a$, we have
\[
\int_a^b\exp\Big(it\log\Big(\frac{t}{u\er}\Big)\Big)
\Big(\frac{t}{2\pi}\Big)^{c-1/2}dt
=(2\pi)^{1-c}u^c\er^{-iu+\pi i/4}\cdot\ind{a,b}(u)+\widetilde E(a,b,u),
\]
where
\[
\ind{a,b}(u)\defeq\begin{cases}
1&\quad\hbox{if $u\in(a,b]$},\\
0&\quad\hbox{otherwise},\\
\end{cases}
\]
and
\[
\widetilde E(a,b,u)\ll a^{c-1/2}+\frac{a^{c+1/2}}{|a-u|+a^{1/2}}
+\frac{b^{c+1/2}}{|b-u|+b^{1/2}}.
\]
\end{lemma}

The next lemma is a variant of 
Conrey, Ghosh, and Gonek \cite[Lemma~1]{ConGhoGon}.

\begin{lemma}\label{lem:ConGhoGon}
Uniformly for $v>0$ and $c\in[\frac{1}{10},2]$, we have
\[
\frac{1}{2\pi i}\int_{c+i}^{c+iT}v^{-s}\cX_\chi(1-s)\,ds
=\begin{cases}
\displaystyle{\frac{\tau(\chi)}{q}\,\e(-v/q)+E(q,T,v)}
&\quad\hbox{if $\frac{q}{2\pi}<v\le \frac{qT}{2\pi}$},\\
E(q,T,v)&\quad\hbox{otherwise},\\
\end{cases}
\]
where
\[
E(q,T,v)\ll\frac{q^{c-1/2}}{v^c}
\bigg(T^{c-1/2}+\frac{T^{c+1/2}}{|T-2\pi v/q|+T^{1/2}}\bigg).
\]
\end{lemma}

\begin{proof}
Using Lemma~\ref{lem:Xxexpansion}, we have
\dalign{
&\frac{1}{2\pi i}\int_{c+i}^{c+iT}v^{-s}\cX_\chi(1-s)\,ds
=\frac{1}{2\pi}\int_1^Tv^{-c-it}\cX_\chi(1-c-it)\,dt\\
&\qquad\quad=\frac{\tau(\chi) q^{c-1}\er^{-\pi i/4}}{2\pi v^c}\bigg(\int_1^T 
\exp\Big(it\log\Big(\frac{qt}{2\pi v\er}\Big)\Big)
\Big(\frac{t}{2\pi}\Big)^{c-1/2}\,dt+O(T^{c-1/2})\bigg),
}
and the result follows by applying Lemma~\ref{lem:gonekLem8} with $u\defeq 2\pi v/q$.
\end{proof}

\subsection{Essential bound}

\begin{lemma}\label{lem:trade}
For any $t\ge 2$, there is a real number $t_*\in[t,t+1]$ such that
\[
\frac{L'}{L}(\sigma\pm it_*,\chi)\ll(\log qt)^2\qquad(-1\le\sigma\le 2).
\]
\end{lemma}

\begin{proof}
See \cite[Lemmas 12.2 and 12.7]{MontVau}.
\end{proof}

\subsection{Conditional results}

\begin{lemma}\label{lem:ultraclean}
The following statements are equivalent:
\begin{itemize}
\item[$(i)$] \text{\rm GRH} is true;
\item[$(ii)$] \text{\rm RH} is true, and
for any primitive character $\psi\ne\one$,
any $\euB\in C_c^\infty(\R^+)$, and any $\eps>0$, we have
\be\label{eq:byebirdy}
\sum_n\Lambda(n)\psi(n)\euB(n/X)
\,\mathop{\ll}\limits_{\psi,\euB,\eps}\, X^{1/2+\eps};
\ee
\item[$(iii)$] \text{\rm RH} is true, and
for any nonprincipal character $\psi$,
any $\euB\in C_c^\infty(\R^+)$, and any $\eps>0$, 
the bound \eqref{eq:byebirdy} holds.
\end{itemize}
\end{lemma}

\begin{proof}
The equivalence $(i)\Longleftrightarrow(ii)$
is the content of \cite[Lemma 2.2]{Banks2}, and the
implication $(iii)\Longrightarrow(ii)$ is obvious.
Using the simple bound
\be\label{eq:lambdarules}
\ssum{n\le N\\(n,M)\ne 1}\Lambda(n)\ll\log M\cdot\log N
\qquad(M,N\ge 1),
\ee
the implication $(ii)\Longrightarrow(iii)$ is immediate.
\end{proof}

\begin{lemma}\label{lem:aloevera}
Under \text{\rm GRH}, for any $\xi\in\Q^+$,
 $\euB\in C_c^\infty(\R^+)$, and  $\eps>0$, we have
\be\label{eq:meowing}
\frac{\tau(\ochi)}{q}\sum_n\Lambda(n)\chi(n)\e(-n\xi/q)\euB(n/qX)
=C_\euB\widetilde C_{\chi,\xi} X
+O_{\chi,\xi,\euB,\eps}(X^{1/2+\eps}),
\ee
where $C_\euB$ and $\widetilde C_{\chi,\xi}$ are defined
by \eqref{eq:Cconsts}.
\end{lemma}

\begin{proof}
Let $\xi\defeq h/k$ with $h,k>0$ and $(h,k)=1$.
Using \eqref{eq:lambdarules}, we see that
the sum in \eqref{eq:meowing} is equal to
\[\sum_{(n,qk)=1}\Lambda(n)\chi(n)\e(-nh/qk)\euB(n/qX)
+O_{\chi,\xi,\euB}(1).
\]
and the latter sum can be expressed as
\dalign{
&\ssum{a\bmod qk\\(a,qk)=1}\e(-ah/qk)\chi(a)
\sum_{n\equiv a\bmod qk}\Lambda(n)\euB(n/qX)\\
&\qquad\qquad=\frac{1}{\phi(qk)}\ssum{a\bmod qk\\(a,qk)=1}
\e(-ah/qk)\chi(a)\sum_{\psi\bmod qk}\opsi(a)
\sum_n\Lambda(n)\psi(n)\euB(n/qX),
}
where the middle sum runs over all characters $\psi$ modulo $qk$.
By Lemma~\ref{lem:ultraclean}\,$(iii)$
the contribution from all nonprincipal characters $\psi$
is $O_{\chi,\xi,\euB,\eps}(X^{1/2+\eps})$.
On the other hand, for the principal character $\psi_0$,
the contribution is
\[
\frac{C}{\phi(qk)}\ssum{(n,qk)=1}\Lambda(n)\euB(n/qX)
=\frac{C}{\phi(qk)}\sum_n\Lambda(n)\euB(n/qX)+O_{\chi,\xi,\euB}(1),
\]
where we used \eqref{eq:lambdarules} again, and
\[
C\defeq\ssum{a\bmod qk\\(a,qk)=1}\e(-ah/qk)\chi(a).
\]
By \cite[Theorem~9.12]{MontVau},
\[
C=\begin{cases}
\ochi(-h)\chi(k)\mu(k)\tau(\chi)
&\quad\hbox{if $(h,q)=1$},\\
0&\quad\hbox{otherwise},\\
\end{cases}
\]
and by \cite[Lemma 2.3]{Banks2},
\[
\sum_n\Lambda(n)\euB(n/qX)=C_\euB\cdot qX+O_{\chi,\euB}(X^{1/2}\log^2X).
\]
To finish the proof, observe that
\[
\frac{\tau(\ochi)}{q}\cdot\frac{C}{\phi(qk)}\cdot q
=\widetilde C_{\chi,\xi},
\]
which follows from the well known relation
\be\label{eq:tautauq}
\tau(\chi)\tau(\ochi)=\chi(-1)q
\ee
for the Gauss sums defined in \eqref{eq:red}.
\end{proof}

\section{Twisting the von Mangoldt function}
\label{sec:vonMangoldt-twist}

The results of this section are \emph{unconditional}.

\begin{theorem}\label{thm:vonMangoldt-twist}
For any $\xi\in\R^+$ and $T\ge 2q^2$, we have
\[
\ssum{\rho=\beta+i\gamma\\0<\gamma\le T}\xi^{-\rho}\cX_\ochi(1-\rho)
+\frac{\tau(\ochi)}{q}\sum_{n\le qT/2\pi\xi}\Lambda(n)\chi(n)\e(-n\xi/q)
\,\mathop{\ll}\limits_\xi\,(qT)^{1/2}\log^2T.
\]
where the first sum runs over complex zeros
$\rho=\beta+i\gamma$ of $L(s,\chi)$ $($counted with multiplicity$)$.
\end{theorem}

\begin{proof}
For any $u>0$, let
\[
\Sigma_1(u)\defeq\ssum{\rho=\beta+i\gamma\\0<\gamma\le u}
\xi^{-\rho}\cX_\ochi(1-\rho),\qquad
\Sigma_2(u)\defeq\frac{\tau(\ochi)}{q}
\ssum{n\le qu/2\pi\xi}\Lambda(n)\chi(n)\e(-n\xi/q).
\]
Our goal is to show that
\be\label{eq:Sigsum}
\Sigma_1(T)+\Sigma_2(T)
\,\mathop{\ll}\limits_\xi\,(qT)^{1/2}\log^2T.
\ee

According to Lemma~\ref{lem:trade}, there is a number
$t_\circ\in[2,3]$ such that
\be\label{eq:LD-horiz-low}
\frac{L'}{L}(\sigma\pm it_\circ,\chi)\ll\log^22q\qquad(-1\le\sigma\le 2).
\ee
Let $t_\circ$ be fixed in what follows. Note that
\be\label{eq:gosh1}
\Sigma_1(t_\circ)\,\mathop{\ll}\limits_\xi\,q^{3/2}\log 2q,
\ee
since $|\xi^{-\rho}\cX_\ochi(1-\rho)|\ll_\xi q^{1/2}$ for all zeros
$\rho=\beta+i\gamma$ with $0<\gamma\le t_\circ$
(see Lemma~\ref{lem:Xxexpansion}),
and there are only $O(q\log 2q)$ such zeros
(see Siegel~\cite[Theorem~III]{Siegel}). By the Chebyshev
bound, we also have
\be\label{eq:gosh2}
\Sigma_2(t_\circ)\,\mathop{\ll}\limits_\xi\,q^{1/2}.
\ee

Using Lemma~\ref{lem:trade} again,
for any $T_*\ge 2q^2$, there is a number $T\in[T_*,T_*+1]$ for which
\be\label{eq:LD-horiz}
\frac{L'}{L}(\sigma\pm iT,\chi)\ll\log^2qT
\ll\log^2T\qquad(-1\le\sigma\le 2).
\ee
To prove \eqref{eq:Sigsum} in general, one can assume without
loss of generality that $T$ has the property \eqref{eq:LD-horiz}.
Indeed, let $T_*\ge 2q^2$ be
arbitrary, and suppose $T\in[T_*,T_*+1]$ satisfies both \eqref{eq:Sigsum}
and \eqref{eq:LD-horiz}. By Lemma~\ref{lem:Xxexpansion}, the bound
$\cX_\ochi(1-\rho)\ll (q\gamma)^{1/2}$
holds uniformly for all complex zeros $\rho=\beta+i\gamma$ of $L(s,\chi)$
with $\gamma\ge 1$; therefore,
\[
\big|\Sigma_1(T_*)-\Sigma_1(T)\big|\le
\ssum{\rho=\beta+i\gamma\\T_*<\gamma\le T_*+1}
\big|\xi^{-\rho}\cX_\ochi(1-\rho)\big|
\,\mathop{\ll}\limits_\xi\,(qT_*)^{1/2}\log T_*
\]
since the number of zeros with $T_*<\gamma\le T_*+1$
is at most $O(\log T_*)$. Moreover,
\[
\big|\Sigma_2(T_*)-\Sigma_2(T)\big|\le\frac{|\tau(\ochi)|}{q}
\ssum{qT_*/2\pi\xi<n\le q(T_*+1)/2\pi\xi}\big|\Lambda(n)\chi(n)\e(-n\xi/q)\big|\,\mathop{\ll}\limits_\xi\,q^{1/2}\log T_*.
\]
Combining the preceding two bounds and \eqref{eq:Sigsum}, we deduce that
\[
\Sigma_1(T_*)+\Sigma_2(T_*)
\,\mathop{\ll}\limits_\xi\,(qT_*)^{1/2}\log^2T_*,
\]
which shows that \eqref{eq:Sigsum} holds with $T_*$ in place of $T$.

From now on, we can assume that $T$ satisfies \eqref{eq:LD-horiz}.
Put $c\defeq 1+\frac{1}{\log qT}$,
and let $\cC$ be the following rectangle in $\C$:
\[
c+it_\circ
~~\longrightarrow~~c+iT
~~\longrightarrow~~-1+iT
~~\longrightarrow~~-1+it_\circ
~~\longrightarrow~~c+it_\circ.
\]
Note that, by \eqref{eq:LD-horiz-low} and \eqref{eq:LD-horiz},
neither $t_\circ$ nor $T$ is the ordinate of a zero of $L(s,\chi)$.
By Cauchy's theorem, 
\dalign{
&\Sigma_1(T)-\Sigma_1(t_\circ)
=\frac{1}{2\pi i}\oint_{\cC}\frac{L'}{L}(s,\chi)\,\xi^{-s}\cX_\ochi(1-s)\,ds\\
&\qquad\qquad=\frac{1}{2\pi i}\bigg(\int_{c+it_\circ}^{c+iT}
+\int_{c+iT}^{-1+iT}
+\int_{-1+iT}^{-1+it_\circ}
+\int_{-1+it_\circ}^{c+it_\circ}\bigg)
\frac{L'}{L}(s,\chi)\,\xi^{-s}\cX_\ochi(1-s)\,ds\\
&\qquad\qquad=I_1+I_2+I_3+I_4\quad\text{(say)},
}
and so by \eqref{eq:gosh1}, we have
\be\label{eq:Ibegin}
\Sigma_1(T)=I_1+I_2+I_3+I_4+O_\xi(q^{3/2}\log 2q).
\ee
We estimate the four integrals $I_j$ separately.

Noting that
\[
\frac{L'}{L}(s,\chi)=-\sum_n\frac{\Lambda(n)\chi(n)}{n^s}
\qquad(\sigma>1),
\]
and applying Lemma~\ref{lem:ConGhoGon}, we get that
\dalign{
I_1&=-\sum_n\Lambda(n)\chi(n)\cdot
\frac{1}{2\pi i}\int_{c+it_\circ}^{c+iT}
(n\xi)^{-s}\cX_\ochi(1-s)\,ds\\
&=-\frac{\tau(\ochi)}{q}
\sum_{qt_\circ/2\pi\xi<n\le qT/2\pi\xi}\Lambda(n)\chi(n)\e(-n\xi/q)
-\sum_n\Lambda(n)\chi(n)E(q,T,n\xi)\\
&=\Sigma_2(t_\circ)-\Sigma_2(T)
+O\bigg(\frac{(qT)^{c-1/2}}{\xi^c}(E_1+E_2)\bigg),
}
where
\[
E_1\defeq\sum_n\frac{\Lambda(n)}{n^c}
\mand
E_2\defeq\sum_n\frac{\Lambda(n)}{n^c}
\frac{T}{|T-2\pi n\xi/q|+T^{1/2}}.
\]
Note that $(qT)^{c-1/2}\asymp (qT)^{1/2}$, and recall that
$\Sigma_2(t_\circ)$ satisfies \eqref{eq:gosh2}.
Also,
\be\label{eq:brilliant}
E_1=-\frac{\zeta'}{\zeta}(c)\ll \frac{1}{c-1}=\log qT\ll\log T.
\ee
To bound $E_2$, we split the integers $n\ge 2$ into
three disjoint sets:
\dalign{
S_1&\defeq\{n\ge 2:|T-2\pi n\xi/q|>\tfrac12T\},\\
S_2&\defeq\{n\ge 2:|T-2\pi n\xi/q|\le T^{1/2}\},\\
S_3&\defeq\{n\ge 2:T^{1/2}<|T-2\pi n\xi/q|\le \tfrac12T\}.
}
Using \eqref{eq:brilliant}, we have
\[
\sum_{n\in S_1}\frac{\Lambda(n)}{n^c}\frac{T}{|T-2\pi n\xi/q|+T^{1/2}}
\ll\log T,
\]
which is acceptable. For each $n\in S_2$, we have
\[
n\,\mathop{\asymp}\limits_\xi\,qT,\qquad
\frac{\Lambda(n)}{n^c}
\,\mathop{\ll}\limits_\xi\,\frac{\log qT}{(qT)^c}
\ll\frac{\log T}{qT},\qquad
\frac{T}{|T-2\pi n\xi/q|+T^{1/2}}
\asymp T^{1/2}.
\]
Since $|S_2|\ll_\xi qT^{1/2}$, it follows that
\[
\sum_{n\in S_2}\frac{\Lambda(n)}{n^c}\frac{T}{|T-2\pi n\xi/q|+T^{1/2}}
\,\mathop{\ll}\limits_\xi\,\log T.
\]
Similarly,
\[
\sum_{n\in S_3}\frac{\Lambda(n)}{n^c}\frac{T}{|T-2\pi n\xi|+T^{1/2}}
\,\mathop{\ll}\limits_\xi\,\frac{\log T}{q}
\sum_{n\in S_3}\frac{1}{|T-2\pi n\xi/q|+T^{1/2}}.
\]
The last sum is bounded by
\[
\ll\sum_{T^{1/2}\le k\le \frac12T}
\ssum{n\ge 1\\k<|T-2\pi n\xi/q|\le k+1}\frac{1}{|T-2\pi n\xi/q|+T^{1/2}}
\,\mathop{\ll}\limits_\xi\,
\sum_{T^{1/2}\le k\le \frac12T}\frac{q}{k}
\ll q\log T,
\]
hence
\[
\sum_{n\in S_3}\frac{\Lambda(n)}{n^c}\frac{T}{|T-2\pi n\xi|+T^{1/2}}
\,\mathop{\ll}\limits_\xi\,\log^2T.
\]
Putting everything together, we deduce that
\[
I_1=-\Sigma_2(T)+O_\xi\big((qT)^{1/2}\log^2T\big).
\]

Next, by Lemma~\ref{lem:Xxexpansion}, we have the uniform bound
\be\label{eq:oralb}
\cX_\ochi(1-\sigma-it)\ll (qt)^{\sigma-1/2}
\qquad(-1\le \sigma\le c,~t\ge 1).
\ee
Recalling \eqref{eq:LD-horiz}, it follows that
\[
I_2=-\frac{1}{2\pi i}\int_{-1+iT}^{c+iT}
\frac{L'}{L}(s,\chi)\,\xi^{-s}\cX_\ochi(1-s)\,ds
\,\mathop{\ll}\limits_\xi\,(qT)^{1/2}\log^2T.
\]
Similarly, by \cite[Lemmas~12.4 and 12.9]{MontVau}, we have the bound
\[
\frac{L'}{L}(-1+it,\chi)\ll\log 2qt\qquad(1\le t\le T).
\]
Taking $\sigma\defeq -1$ in \eqref{eq:oralb}, we get that
\[
I_3=-\frac{1}{2\pi i}\int_{-1+it_\circ}^{-1+iT}
\frac{L'}{L}(s,\chi)\,\xi^{-s}\cX_\ochi(1-s)\,ds
\ll\xi^{-1}\int_{t_\circ}^T \frac{\log 2qt}{(qt)^{3/2}}\,dt
\,\mathop{\ll}\limits_\xi\,1.
\]
Finally, using \eqref{eq:oralb} and \eqref{eq:LD-horiz-low},
we get that
\[
I_4=\frac{1}{2\pi i}\int_{-1+it_\circ}^{c+it_\circ}
\frac{L'}{L}(s,\chi)\,\xi^{-s}\cX_\ochi(1-s)\,ds
\,\mathop{\ll}\limits_\xi\,q^{1/2}\log^22q.
\]

Combining \eqref{eq:Ibegin} and the above estimates for the integrals $I_j$,
we obtain \eqref{eq:Sigsum}, and the proof is complete.
\end{proof}

Let $\log_+u\defeq\max\{\log u,1\}$ for all $u>0$.

\begin{corollary}\label{cor:vonMangoldt-twist}
For any $\xi\in\R^+$ and $T>0$, we have
\[
\ssum{\rho=\beta+i\gamma\\0<\gamma\le T}\xi^{-\rho}\cX_\ochi(1-\rho)
+\frac{\tau(\ochi)}{q}\sum_{n\le qT/2\pi\xi}\Lambda(n)\chi(n)\e(-n\xi/q)
\,\mathop{\ll}\limits_{\chi,\xi}\,T^{1/2}\log_+^2T.
\]
where the first sum runs over complex zeros
$\rho=\beta+i\gamma$ of $L(s,\chi)$ $($counted with multiplicity$)$.
\end{corollary}

\begin{corollary}\label{cor:vonMangoldt-twist}
For any $\xi\in\R^+$, $\euB\in C_c^\infty(\R^+)$, and $X>0$,
we have
\[
\ssum{\rho=\beta+i\gamma}\xi^{-\rho}
\cX_\ochi(1-\rho)\euB\Big(\frac{\gamma}{2\pi\xi X}\Big)
+\frac{\tau(\ochi)}{q}
\sum_n\Lambda(n)\chi(n)\e(-n\xi/q)\euB(n/qX)
\,\mathop{\ll}\limits_{\chi,\xi,\euB}\,X^{1/2}\log_+^2X
\]
where the first sum runs over the complex zeros
$\rho=\beta+i\gamma$ of $L(s,\chi)$ $($counted with multiplicity$)$.
\end{corollary}

\begin{proof}
Let $\Sigma_1$ and $\Sigma_2$ be the step functions defined in 
the proof of Theorem~\ref{thm:vonMangoldt-twist},
and put
\[
\Sigma_3(u)\defeq\Sigma_2(2\pi\xi u/q)=\frac{\tau(\ochi)}{q}
\ssum{n\le u}\Lambda(n)\chi(n)\e(-n\xi/q).
\]
Applying Corollary~\ref{cor:vonMangoldt-twist}
with $T\defeq 2\pi\xi Xu$, we have
\be\label{eq:lawnmower}
\Sigma_1(2\pi\xi Xu)+\Sigma_2(2\pi\xi Xu)
\,\mathop{\ll}\limits_{\chi,\xi}\,(Xu)^{1/2}\log_+^2(Xu).
\ee

Using Riemann-Stieltjes integration, we have
\dalign{
&\ssum{\rho=\beta+i\gamma}\xi^{-\rho}
\cX_\ochi(1-\rho)\euB\Big(\frac{\gamma}{2\pi\xi X}\Big)
=\int_0^\infty\euB\Big(\frac{u}{2\pi\xi X}\Big)\,d\Sigma_1(u)\\
&\qquad\qquad=\int_0^\infty\euB(u)\,d\Sigma_1(2\pi\xi Xu)
=-\int_0^\infty\euB'(u)\Sigma_1(2\pi\xi Xu)\,du,
}
and
\dalign{
&\frac{\tau(\ochi)}{q}\sum_n\Lambda(n)\chi(n)\e(-n\xi/q)\euB(n/qX)
=\int_0^\infty\euB(u/qX)\,d\Sigma_3(u)\\
&\qquad\qquad=\int_0^\infty\euB(u)\,d\Sigma_3(qXu)
=-\int_0^\infty\euB'(u)\Sigma_3(qXu)\,du\\
&\qquad\qquad=-\int_0^\infty\euB'(u)\Sigma_2(2\pi\xi Xu)\,du.
}
Summing these expressions and using \eqref{eq:lawnmower},
the result follows.
\end{proof}

\section{Proof of Theorem~\ref{thm:RHvsGRH}}

Throughout the proof, we fix $q\ge 1$ and a primitive character
$\chi$ modulo $q$, and thus, any implied constants in the
symbols $\ll$, $O$, etc., are independent of $\chi$ (and $q$).
In particular, $q\ll 1$.

First, assume GRH. Then ${\tt GRH}[\chi]$ is true. 
For any $\xi\in\Q^+$, $\euB\in C_c^\infty(\R^+)$, and $\eps>0$,
we have by Lemma~\ref{lem:aloevera}:
\be\label{eq:2112}
\frac{\tau(\ochi)}{q}\sum_n\Lambda(n)\chi(n)\e(-n\xi/q)\euB(n/qX)
=C_\euB\widetilde C_{\chi,\xi} X
+O_{\xi,\euB,\eps}(X^{1/2+\eps}).
\ee
Also, Corollary~\ref{cor:vonMangoldt-twist} states that
\[
\ssum{\rho=\frac12+i\gamma}\xi^{-\rho}
\cX_\ochi(1-\rho)\euB\Big(\frac{\gamma}{2\pi\xi X}\Big)
+\frac{\tau(\ochi)}{q}
\sum_n\Lambda(n)\chi(n)\e(-n\xi/q)\euB(n/qX)
\,\mathop{\ll}\limits_{\xi,\euB,\eps}\,X^{1/2+\eps}.
\]
Combining these results, we obtain \eqref{eq:superbound},
which proves the validity of ${\tt GRH}^\dagger[\chi]$.

Conversely, suppose ${\tt GRH}^\dagger[\chi]$ is true.
Then ${\tt GRH}[\chi]$ is true, and for any
$\xi\in\Q^+$, $\euB\in C_c^\infty(\R^+)$, and $\eps>0$, 
we have
\[
\ssum{\rho=\frac12+i\gamma}\xi^{-\rho}
\cX_\ochi(1-\rho)\euB\Big(\frac{\gamma}{2\pi\xi X}\Big)
+C_\euB\widetilde C_{\chi,\xi}X
\,\mathop{\ll}\limits_{\xi,\euB,\eps}\,X^{1/2+\eps}
\]
Using Corollary~\ref{cor:vonMangoldt-twist} again, it follows that
\eqref{eq:2112} holds for our character $\chi$.

Now let $\psi$ be an arbitrary primitive character of
modulus $\tilde q\ge 1$, and let $\vartheta$ be the character modulo
$q\tilde q$  defined by $\vartheta(n)\defeq\psi(n)\ochi(n)$
for all $n\in\Z$. Put
\be\label{eq:W1}
\cW\defeq\tau(\otheta)\sum_n\Lambda(n)\psi(n)\euB(n/X).
\ee
Note that (see, e.g., \cite[Theorem~9.10]{MontVau})
\be\label{eq:tauthetareln}
\tau(\otheta)=\chi(\tilde q)\mu(\tilde q)\tau(\chi).
\ee
Using \eqref{eq:lambdarules}, we have
\[
\cW=\tau(\otheta)\sum_{(n,q\tilde q)=1}
\Lambda(n)\vartheta(n)\chi(n)\euB(n/X)
+O_{\psi,\euB}(1).
\]
When $(n,q\tilde q)=1$, we have (see \cite[Theorem~9.5]{MontVau})
\[
\vartheta(n)\tau(\otheta)
=\sum_{a\bmod q\tilde q}\otheta(a)\e(an/q\tilde q)
=\ssum{1\le a\le q\tilde q\\(a,q\tilde q)=1}
\otheta(-a)\e(-an/q\tilde q),
\]
and so we get that
\[
\cW=\ssum{1\le a\le q\tilde q\\(a,q\tilde q)=1}\otheta(-a)\sum_{(n,q\tilde q)=1}
\Lambda(n)\chi(n)\e(-an/q\tilde q)\euB(n/X)
+O_{\psi,\euB}(1).
\]
Applying \eqref{eq:2112} with $X/q$ instead of $X$,
and $\xi\defeq a/\tilde q$
for $1\le a\le q\tilde q$ with $(a,q\tilde q)=1$, we derive that
\[
\cW=\frac{C_\euB X}{\tau(\ochi)}\ssum{1\le a\le q\tilde q\\(a,q\tilde q)=1}\otheta(-a)
\widetilde C_{\chi,a/\tilde q}
+O_{\psi,\euB,\eps}(X^{1/2+\eps}).
\]
Finally, when $(a,q\tilde q)=1$, we have
\[
\otheta(-a)\widetilde C_{\chi,\xi}
=\opsi(-a)\chi(-a)\cdot
\frac{\ochi(a)\chi(\tilde q)\mu(\tilde q)\,q}{\phi(q\tilde q)}
=\opsi(-a)\cdot
\frac{\tau(\otheta)\tau(\ochi)}{\phi(q\tilde q)},
\]
where we used the relations \eqref{eq:tautauq} and \eqref{eq:tauthetareln}
in the last step; consequently,
\be\label{eq:W2}
\cW=C_\euB\,
\frac{\tau(\otheta)}{\phi(q\tilde q)}\,X
\ssum{1\le a\le q\tilde q\\(a,q\tilde q)=1}\opsi(-a)
+O_{\psi,\euB,\eps}(X^{1/2+\eps}).
\ee

Comparing \eqref{eq:W1} and \eqref{eq:W2}, we conclude that
the estimates
\be\label{eq:RHcrit}
\sum_n\Lambda(n)\euB(n/X)=
C_\euB X+O_{\euB,\eps}(X^{1/2+\eps})
\ee
and
\[
\sum_n\Lambda(n)\psi(n)\euB(n/X)
\,\mathop{\ll}\limits_{\psi,\euB,\eps}\,X^{1/2+\eps}
\qquad(\psi\ne\one)
\]
hold for all choices of $\euB\in C_c^\infty(\R^+)$ and $\eps>0$.
Since \eqref{eq:RHcrit} implies RH,
this verifies statement Lemma~\ref{lem:ultraclean}\,$(ii)$,
and by the lemma, GRH follows.

\end{document}